\documentclass[11pt]{article}

\usepackage{amsmath,amsthm,verbatim,amssymb,amsfonts,amscd, graphicx,mathrsfs}
\usepackage{ upgreek }

\usepackage{graphics}

\theoremstyle{plain}

\newtheorem{theorem}{Theorem}

\newtheorem{lemma}{Lemma}

\newtheorem{conjecture}{Conjecture} 

\theoremstyle{definition}

\newtheorem{definition}{Definition}

\newcommand{\FF}{\mathcal{F}}

\newcommand{\GG}{\mathcal{G}}

\newcommand{\NN}{\mathbb{N}}

\begin{document}
\title{New results on simplex-clusters in set systems}
\author{Gabriel Currier \footnote{Department of Mathematics, University of British Columbia, Vancouver, Canada. currierg@math.ubc.ca}}
\maketitle
\abstract{\noindent A $d$-simplex is defined to be a collection $A_1,\dots,A_{d+1}$ of subsets of size $k$ of $[n]$ such that the intersection of all of them is empty, but the intersection of any $d$ of them is non-empty. Furthermore, a $d$-cluster is a collection of $d+1$ such sets with empty intersection and union of size $\le 2k$, and a $d$-simplex-cluster is such a collection that is both a $d$-simplex and a $d$-cluster. The Erd\H{o}s-Chv\'{a}tal $d$-simplex Conjecture from 1974 states that any family of $k$-subsets of $[n]$ containing no $d$-simplex must be of size no greater than $ {n -1 \choose k-1}$. In 2011, Keevash and Mubayi extended this conjecture by hypothesizing that the same bound would hold for families containing no $d$-simplex-cluster. In this paper, we resolve Keevash and Mubayi's conjecture for all $4 \le d+1 \le k$ and $n\ge 2k-d+2$, which in turn resolves all remaining cases of the Erd\H{o}s-Chv\'{a}tal Conjecture except when $n$ is very small (i.e. $n < 2k-d+2$).}
\section{Introduction}
\noindent For positive integers $n,k$ we define $[n] := \{1,\dots,n\}$ and use ${X \choose k}$ to denote the set of all $k$-element subsets of a set $X$. Furthermore, we refer to a set $\FF \subseteq {[n] \choose k}$ as a family, and if every element of $\FF$ contains some $S \in {[n] \choose s}$, we say that $\FF$ is an $s$-star centered at $S$. If $S = \{x\}$, we say simply that $\FF$ is a star centered at $x$. Observe that $s$-stars can be no bigger than ${n-s \choose k-s}$. The following theorem, commonly known as the Erd\H{o}s-Ko-Rado (EKR) theorem, is a foundational result in extremal combinatorics.

\begin{theorem}\label{oldekr}
Let $n \ge 2k$ and suppose $\FF \subseteq {[n] \choose k}$. Furthermore, if $A \cap B \neq \emptyset$ for all $A,B \in \FF$, then $$|\FF| \le {n-1 \choose k-1},$$ where, if $n > 2k$, equality holds only if $\FF$ is a maximum-sized star.
\end{theorem}

\noindent Families $\FF$ that satisfy this condition are sometimes referred to as pairwise intersecting. Since its original publication in 1961 \cite{ekr}, EKR has seen numerous applications and has been proven using a wide array of different combinatorial and algebraic techniques. It has even generated a whole subfield of extremal combinatorics known as intersection problems, in which one considers the maximum size of a family with a certain forbidden subfamily, where the subfamily is defined according to some intersection or union constraints. For an introduction to this field, we direct readers to a recent survey of Frankl and Tokushige \cite{surv}. One of the more heavily-studied problems in this area involves the notion of a $d$-simplex.

\begin{definition}
A \emph{$d$-simplex} is a set of $d+1$ elements $\{A_1,A_2,\ldots ,A_{d+1}\}$ of $\binom{[n]}{k}$ with the properties
that $\bigcap_{i=1}^{d+1}A_i=\emptyset$ and for any $1 \le j \le d+1$ we have $\bigcap_{i \neq j }A_i\ne\emptyset$.
\end{definition}
\noindent In 1971, Erd\H{o}s conjectured that a family $\FF \subseteq {[n] \choose k}$ that contains no $2$-simplex (also known as a triangle) must adhere to the same bound of $|\FF| \le {n-1 \choose k-1}$. In 1974, Chv\'{a}tal extended this conjecture to the following.

\begin{conjecture}[\cite{chva}]\label{simpconj}
Let $3 \le d+1 \le k$ and $n \ge \frac{d+1}{d}k$, and suppose $\FF \subseteq {[n] \choose k}$ contains no $d$-simplex. Then $$|\FF| \le {n-1 \choose k-1},$$ with equality only if $\FF$ is a star.
\end{conjecture}

\noindent In the same paper Chv\' atal also resolved the case of $k = d+1$. Conjecture \ref{simpconj} is now sometimes referred to as the Erd\H os-Chv\' atal simplex conjecture, and since its inception there have been a number of partial results. The conjecture was first proven for a wide range of $n,k$ and $d$ by Frankl in \cite{firstfran}, and the case of $n > n_0(k,d)$ was resolved by Frankl and F\"{u}redi in \cite{franfur2}. In 2005, Mubayi and Verstra\"{e}te solved completely the case of $d =2$ \cite{vermub}, and in 2010, Keevash and Mubayi solved the case where both $k/n$ and $n/2-k$ are bounded away from zero \cite{kevmub}. Very recently, Keller and Lifshitz showed that the conjecture holds for all $n > n_0(d)$ \cite{kellif}. A related notion, known as a $d$-cluster, is defined as follows.
\begin{definition} A \emph{$d$-cluster} is a set of $d+1$ elements $\{A_1,A_2,\ldots ,A_{d+1}\}$ of $\binom{[n]}{k}$ with the properties
that $\bigcap_{i=1}^{d+1}A_i=\emptyset$ and $|\bigcup_{i=1}^{d+1}A_i|\le 2k$. If $\{A_1,\dots,A_{d+1}\}$ is both a $d$-simplex and a $d$-cluster, we say that it is a \emph{$d$-simplex-cluster.}
\end{definition}
\noindent As with simplices, it was conjectured \cite{3mub} that a family $\FF \subseteq {[n] \choose k}$ containing no $d$-cluster would have to obey the bound $|\FF| \le {n-1 \choose k-1}$. This problem also had a long history (see \cite{franfur,3mub,4mub,rammub,kevmub} for some of the more significant developments) and was completely resolved recently in a paper of the author \cite{me}. In 2010, Keevash and Mubayi extended both conjectures by hypothesizing that the same bound would hold for any $\FF \subseteq {[n] \choose k}$ containing no $d$-simplex-cluster, and very recently Lifshitz answered their question in the affirmative for all $n > n_0(d)$ in \cite{sololif}.
\\\\
\noindent The primary goal of this paper is to show that that Keevash and Mubayi's conjecture holds for all $n \ge 2k-d+2$ when $d \ge 3$.

\begin{theorem}\label{mainthm}
Suppose that $4 \le d+1 \le k$ and $n \ge 2k-d+2$, and that $\FF \subseteq {[n] \choose k}$ contains no $d$-simplex-cluster. Then
$$|\FF| \le {n-1 \choose k-1},$$
with equality only if $\FF$ is a maximum sized star.
\end{theorem}
\noindent A family containing no $d$-simplex must also contain no $d$-simplex-cluster. Thus, we get as an immediate corollary (when combined with results from \cite{vermub} for the case $d=2$) that Conjecture \ref{simpconj} holds for all values of $d,k$ and $n$ except for the very small values of $n$, where $n < 2k-d+2$.
\\\\
\noindent We take a moment here to discuss the range $n < 2k$. Intersection problems of this type tend to have a slightly different flavor when considered for these values of $n$. There are sometimes obvious reasons for this - the Erd\H os-Ko-Rado theorem in its original form, for example, does not make much sense when considered in this context because any two $k$-sets will intersect. As another example, the problem of clusters is different in this range because the union condition holds automatically, so a $d$-cluster-free family is simply a family with no $d+1$ sets that have empty intersection. There is, however, some history of results for problems of this type in the range $n < 2k$. The most notable example is perhaps the Complete Intersection Theorem of Ahlswede and Khachatrian \cite{ahlk}, which gives us a characterization of families where all elements intersect each other in at least $t$ places. For simplices, there are two results known to the author. The first is for $d=2$, when the full range of $n \ge 3k/2$ was shown in \cite{vermub}. The other example (and the only for general $d$) is from \cite{spernfam}, where the case of $n < k\frac{d}{d-1}$ is resolved.
\\\\
In our proof of Theorem \ref{mainthm}, we will use as one of our primary tools the following theorem, which was used in \cite{me} to resolve the question of $d$-clusters. To our knowledge, this theorem was shown first by Borg in \cite{borg1} using shadow techniques, and by Borg and Leader in \cite{borg2} using Katona's cycle method. It can also be seen as a specific case of a more general theorem from \cite{frankup} (Theorem $9$ with $c = \frac{n-k}{k}$). We include the proof for completeness.
\\\\
\begin{theorem}\label{modekr}
Let $n \ge 2k$, and suppose $\FF^* \subseteq \FF \subseteq {[n] \choose k}$ are such that $A,B \in \FF^*$ for any $A,B \in \FF$ with $A \cap B = \emptyset$. Then
$$|\FF| \le {n-1 \choose k-1} + \frac{n-k}{n}|\FF^*|,$$
where, for $n > 2k$, equality is achieved only if $\FF^* = \FF = {[n] \choose k}$ or if $\FF^*= \emptyset$ and $\FF$ is a maximum sized star.
\end{theorem}
\noindent Note that this theorem is itself a generalization of EKR; if one sets $\FF^* = \emptyset$, then Theorem \ref{modekr}  says simply that a pairwise intersecting family has size at most ${n-1 \choose k-1}$. 

\begin{proof}
We will proceed by the Katona cycle method \cite{katcyc}. First, we let $C(n)$ denote the set of all cyclic permutations on $n$ elements. Then, if we have $(a_0,\dots,a_{n-1}) = \sigma \in C(n)$ and $\GG \subset {[n] \choose k}$, we define (with all subscripts henceforth taken mod $n$)
$$S_\sigma(\GG) := \{A \in \GG : A = \{a_i,a_{i+1},\dots,a_{i+(k-1)}\} \text{ for some } i \in [0,n-1]\}.$$
Observe trivially that $|S_\sigma(\GG)| \le n$. Furthermore, for any such $A = \{a_i,\dots,a_{i+(k-1)}\}$, we say that $A$ has \emph{starting point} $i$ in $\sigma = (a_0,\dots,a_{n-1})$. Now, we wish to prove the following:
\begin{enumerate}
\item $|S_\sigma(\FF \setminus \FF^*)| \le k$ for all $\sigma \in C(n)$
\item if $S_\sigma(\FF \setminus \FF^*) \neq \emptyset$, then $|S_\sigma(\FF^*)| \le 2(k-|S_\sigma(\FF \setminus\FF^*)|)$ for all $\sigma \in C(n)$
\end{enumerate}
Let $\sigma = (a_0,\dots,a_{n-1})$ as before, suppose $S_\sigma(\FF \setminus \FF^*) \neq \emptyset$, and take $A \in S_\sigma(\FF \setminus \FF^*)$. Furthermore, suppose without loss of generality that $A = \{a_0,\dots,a_{k-1}\}$. Then, let $A' \in (S_\sigma(\FF) \setminus \{A\})$ and observe that since $A \cap A' \neq \emptyset$, it follows that $A'$ has starting point in either $ [n-(k-1),n-1]$ or $[1,k-1]$. Suppose then that we have $A_1,A_2 \in (S_\sigma(\FF) \setminus \{A\})$ with starting points $i_1 \in [n-(k-1),n-1]$ and $(i_1 + k) \in [1,k-1]$ in $\sigma$ respectively. Since $n \ge 2k$ this implies $A_1 \cap A_2 = \emptyset$ and thus that $A_1,A_2 \in \FF^*$. Because only one element of $\FF$ may have a given starting point in $\sigma$, we can combine these facts to get both (i) and (ii). Now, we define subsets of $C(n)$
$$C_j := \{\sigma \in C(n) : |S_\sigma(\FF \setminus \FF^*)| = j\},$$
and using (i) we observe that $C_0,C_1,\dots,C_k$ partition $C(n)$. Using (i) and (ii), and since every  $A \in \FF$ is in $S_\sigma(\FF)$ for precisely $k!(n-k)!$ different $\sigma \in C(n)$ we get
$$|\FF \setminus \FF^*|k!(n-k)! = \sum_{\sigma \in C(n)}|S_\sigma(\FF \setminus \FF^*)| = \sum_{1\le i \le k}|C_i|i$$
and
$$|\FF^*|k!(n-k)! = \sum_{\sigma \in C(n)}|S_\sigma(\FF^*)|  \le n|C_0| + \sum_{1 \le i \le k} |C_i| 2(k-i).$$
Combining these yields
\begin{align}
|\FF^*| + \left(\frac{n}{k}\right)|\FF \setminus \FF^*| &\le \frac{ \Big(n|C_0| + \sum_{i=1}^k 2(k-i)|C_i|\Big) + (n/k)\Big(\sum_{i=1}^k i|C_i|\Big)}{k!(n-k)!} \nonumber\\
&=\frac{n|C_0| + n|C_k| + \sum_{i=1}^{k-1}\frac{in+2k(k-i)}{k}|C_i|}{k!(n-k)!}\label{compline}.
\end{align}
A quick calculation gives us that, for all $1 \le i \le k-1$
\begin{align}
\frac{in + 2k(k-i)}{k} \le \frac{in + n(k-i)}{k}= n, \label{eqline2}
\end{align}
with equality only if $n = 2k$. Combining (\ref{compline}) and (\ref{eqline2}), since $|C_0| + \dots + |C_k| = |C(n)| =  (n-1)!$, we get
\begin{align}
\left(\frac{k-n}{k}\right)|\FF^*| + \left(\frac{n}{k}\right) |\FF|&= |\FF^*| + \Big(\frac{n}{k}\Big)|\FF \setminus \FF^*| \nonumber\\
&\le \frac{n(|C_0| + \dots + |C_k|)}{k!(n-k)!}\nonumber\\
&= \frac{n!}{k!(n-k)!}\nonumber \\
&= {n \choose k}. \nonumber
\end{align}
Dividing both sides by $n/k$ we get our desired inequality. Now, suppose $n > 2k$ and we have equality. Note that in this case we do not have equality in (\ref{eqline2}) and so $C(n) = C_0 \cup C_k$. Furthermore, if we take arbitrary $A,A' \in \FF$, we can easily construct $\sigma \in C(n)$ such that $A,A' \in S_\sigma(\FF)$. Since either $\sigma \in C_0$ or $\sigma \in C_k$, this implies that $A,A' \in \FF^*$ or $A,A' \in (\FF \setminus\FF^*)$. Since $A$ and $A'$ were arbitrary, we get that either $\FF = \FF^*$ or $\FF = (\FF \setminus \FF^*)$. If we assume the former then $|\FF| = |\FF^*| = {n \choose k}$ in which case $\FF = {[n] \choose k}$. For the latter, we get that $|\FF| = |\FF \setminus \FF^*| = {n-1 \choose k-1}$ and $\FF$ is pairwise intersecting, in which case Theorem \ref{oldekr} tells us that $\FF$ is a star. This completes the proof.
\end{proof}

\noindent In the remainder of the paper, we will use the following notation.

\begin{definition}
Suppose $\FF \subseteq {[n] \choose k}$ and we have $A \in \FF$ and $D \subseteq [n]$. Then, we define $\bigtriangledown_\FF (D) \subseteq \FF$ and $\alpha_\FF^i(A) \subseteq A$ to be
$$\bigtriangledown_\FF(D) := \{B \in \FF : D \subseteq B\}$$
and
$$\alpha_\FF^i(A) := \{x \in A : \bigtriangledown_\FF(A \setminus \{x\}) = i\}.$$
\end{definition}
\noindent The first definition is related to the common combinatorial notion of link or trace - that is, $\bigtriangledown_\FF(D)$ is all elements of $\FF$ that contain $D$. The second definition can be thought of as a measure of the removability of the elements of $A \in \FF$. By this we mean that, if we have $x \in \alpha_\FF^i(A)$ for some $i \ge 2$, then we can remove $x$ from $A$ without increasing its size very much - that is, we can find $B \in \FF$ such that $x \notin B$ but $|A \cup B|$ is small. Furthermore, if $i \ge 3$ we have greater flexibility in choosing $B$ that we will leverage later on. These are useful notions because they provide a way to construct $d$-simplex-clusters in a way that is both controlled and relatively easy to count.
\\\\
\noindent The following lemmas will make this more precise. The first shows us that if we have $A,B \in \FF$ such that $A \cap B$ satisfies certain size and removability conditions, then $\FF$ must contain a $d$-simplex-cluster.
\begin{lemma}\label{basiclem}
Suppose $d+1 \le k$ and $n \ge 2k-d$, and let $\FF \subseteq {[n] \choose k}$. Then, if there exist $A,B\in \FF$ such that $A \cap B \in {A \setminus \alpha_\FF^1(A) \choose d} \setminus {\alpha_\FF^2(A) \choose d}$, then $\FF$ must contain a $d$-simplex-cluster. 
\end{lemma}
\begin{proof}
Let $A,B$ be as described, with $A \cap B = \{x_1,\dots,x_{d}\}$, and suppose without loss of generality that $x_{d} \in \alpha_\FF^i(A)$ for some $i \ge 3$.  Then, for all $1 \le j \le d$, since $x_j \in \alpha_\FF^i(A)$ for $i \ge 2$, there exists $B_j \in \FF$ such that $A \cap B_j = A \setminus \{x_j\}$. Note at this point that $B_1,\dots,B_{d}$ may have an element (at most one) of intersection outside of $A$. However, if this is the case, since $x_{d} \in \alpha_\FF^i(A)$ for $i \ge 3$, we can re-choose $B_{d}$ such that the $B_1,\dots,B_{d}$ have empty intersection outside of $A$. We claim that $B,B_1,\dots,B_d$ is a $d$-simplex-cluster. Verifying first the intersection condition gives
$$B \cap B_1 \cap \dots \cap B_{d} =  (A \cap B) \cap B_1 \cap \dots \cap B_{d} = \emptyset, $$
and furthermore
$$|B \cup B_1 \cup \dots \cup B_{d}| \le |A \cup B| + |B_1 \setminus A| + \dots + |B_d \setminus A| \le (2k - d)+d = 2k.$$
Finally, we see that $B_1 \cap \dots \cap B_{d} = A \setminus B \neq \emptyset$ and that $x_j \in \left(B \cap \left( \bigcap_{j \neq i} B_j\right)\right)$ for all $1 \le j \le d$. Thus, $B,B_1,\dots,B_{d}$ is a $d$-simplex-cluster, completing the proof.
\end{proof}
\noindent Thus, our task reduces in some sense to proving that $\alpha_\FF^1(A)$ and $\alpha_\FF^2(A)$ are small for most $A \in \FF$. The following lemma will be used to show this.
\begin{lemma}\label{maxbad} 
Suppose $k < n$ and that $\FF \subseteq {[n] \choose k}$ is such that $|\FF| \ge {n-1 \choose k-1}$. Then
$$\sum_{A \in \FF} |\alpha_\FF^1(A)| + \frac{n-k-1}{2(n-k)}|\alpha_\FF^2(A)| \le {n-1 \choose k-1}.$$
\end{lemma}
\begin{proof}
To start, we define $\FF^C := {[n] \choose k} \setminus \FF$, and observe that
$$|\FF|k + |\FF^C|k = {n \choose k}k = n{n-1 \choose k-1}.$$
Using our assumption that $|\FF| \ge {n-1 \choose k-1}$, we see that
\begin{align}
(n-k){n-1 \choose k-1} &\ge k|\FF^C| \nonumber\\
&= \sum_{A \in \FF^C} \sum_{1 \le i \le (n-k+1)} |\alpha_{\FF^C}^i(A)| \nonumber \\
&\ge \sum_{A \in \FF^C} |\alpha_{\FF^C}^{n-k}(A)| + |\alpha_{\FF^C}^{n-k-1}(A)|\nonumber \\
&=\sum_{A \in \FF} (n-k)|\alpha_\FF^1(A)| + \frac{n-k-1}{2}|\alpha_\FF^2(A)|,\nonumber
\end{align}
which is the desired result.
\end{proof}
\noindent Having shown that $\alpha_\FF^1(A)$ and $\alpha_\FF^2(A)$ are small for most $A \in \FF$, we will want to use this in conjunction with Lemma \ref{basiclem}. However, Lemma \ref{basiclem} is a statement about $d$-subsets of $A$, while Lemma \ref{maxbad} is about single elements of $A$. To bridge the gap between these two results, we use the following counting lemma.
\begin{lemma}\label{rsmall}
Suppose $d+1 \le k < n$ and $\FF \subseteq {[n] \choose k}$. Then, we have
$$\sum_{A \in \FF}\Bigg({k \choose d} - {|A \setminus \alpha_\FF^1(A)| \choose d} + {|\alpha_\FF^2(A)| \choose d}\Bigg) \le {k-1 \choose d-1}\sum_{A \in \FF}\left(|\alpha_\FF^1(A)| + \frac{|\alpha_\FF^2(A)|}{d} \right).$$
\end{lemma}
\begin{proof}
We use here the fact that if $m_1,m_2,\ell \in \NN$, then ${m_1 \choose \ell} - {m_2 \choose \ell} = {m_1-1 \choose \ell-1} + {m_1-2 \choose \ell-1} + \dots + {m_2 \choose \ell-1}$, as well as the fact that $|\alpha_\FF^1(A)|,|\alpha_\FF^2(A)| \le k$ for all $A \in \FF$. This yields
\begin{align}
\sum_{A \in \FF}\Bigg({k \choose d} - {|A \setminus \alpha_\FF^1(A)| \choose d} + {|\alpha_\FF^2(A)| \choose d}\Bigg) & \le \sum_{A \in \FF}\Bigg(\sum_{i = 1}^{|\alpha_\FF^1(A)|}{k-i \choose d-1}+ \frac{|\alpha_\FF^2(A)|}{k}{k \choose d}\Bigg) \nonumber\\
&\le \sum_{A \in \FF} \Bigg(|\alpha_\FF^1(A)|{k-1 \choose d-1} + \frac{|\alpha_\FF^2(A)|}{d}{k-1 \choose d-1}\Bigg),\nonumber
\end{align}
thus completing the proof.
\end{proof}
\noindent We may now proceed with the proof of our main result.

\begin{proof}[Proof of Theorem \ref{mainthm}]
Let $|\FF| \ge {n-1 \choose k-1}$ and suppose $\FF$ contains no $d$-simplex-cluster. Then, for any $D \in {[n] \choose d}$, we define the following subset of $\bigtriangledown_\FF(D)$,
$$\bigtriangledown^*_\FF(D) := \{A \in \bigtriangledown_\FF(D) : D \cap \alpha_\FF^1(A) \neq \emptyset \text{ or } D \subseteq \alpha_\FF^2(A)\}.$$
Now, suppose we have $A_1, A_2 \in \bigtriangledown_\FF(D)$ such that $A_1 \cap A_2 = D$ and observe that by Lemma \ref{basiclem} and the fact that $\FF$ contains no $d$-simplex-cluster, we have $A_1,A_2 \in \bigtriangledown^*_\FF(D)$. Thus, we may apply Theorem \ref{modekr} with $\{(A\setminus D) : A \in \bigtriangledown_\FF(D)\}$ as $\FF$ and $\{(A\setminus D) : A \in \bigtriangledown^*_\FF(D)\}$ as $\FF^*$ to get
\begin{align}|\bigtriangledown_\FF(D)|  \le {n-d-1 \choose k-d-1} + \frac{n-k}{n-d} |\bigtriangledown^*_\FF(D)|. \nonumber\end{align}
Summing over all $D \in {[n] \choose d}$, and using Lemmas \ref{maxbad} and \ref{rsmall}, we obtain
\begin{align}
|\FF|{k \choose d} &= \sum_{D \in {[n] \choose d}} |\bigtriangledown_\FF(D)|  \nonumber\\
 &\le {n-d-1 \choose k-d-1}{n \choose d} + \frac{n-k}{n-d}\sum_{D \in {[n] \choose d}}|\bigtriangledown^*_\FF(D)|\label{impline1}\\
&= {n-d-1 \choose k-d-1}{n \choose d} + \frac{n-k}{n-d}\sum_{A \in \FF}\left({k \choose d} - {|A \setminus \alpha_\FF^1(A)| \choose d} + {|\alpha_\FF^2(A)| \choose d}\right) \nonumber\\
&\le {n-1 \choose k-1}{k \choose d}\frac{n(k-d)}{k(n-d)} + \frac{n-k}{n-d}{k-1 \choose d-1}{n-1 \choose k-1} \label{hardeq}\\
&={n-1 \choose k-1}{k \choose d} \Bigg(\frac{n(k-d)}{k(n-d)} + \frac{d(n-k)}{k(n-d)}\Bigg) \nonumber\\
&= {n-1 \choose k-1} {k \choose d}, \nonumber
\end{align}
which is our desired inequality. Note that in (\ref{hardeq}) we have also used that $d \ge 3$ and $n \ge 2k-d+2 \ge k+3$. Now, suppose that we have equality, and in particular that we have equality in (\ref{impline1}). We wish to show that $\FF$ is a star. To start, for every $1 \le \ell \le k$, we define $\GG_\ell \subseteq {[n] \choose \ell}$ as follows
$$\GG_\ell := \{D \in {[n] \choose \ell} : \bigtriangledown_\FF(D) = \bigtriangledown_{[n] \choose k}(D)\}.$$
The proof will proceed as follows: we will start by showing that $|\GG_d| \ge {n-1 \choose d-1}$ and use this to show that $|\GG_{d+1}| \ge {n-1 \choose d}$. Then, we will show that $\GG_{d+1}$ is $d$-simplex-free, and use this to show that it is a star. This will show by extension that $\FF$ is a star.
\\\\
We start by showing that $|\GG_d| \ge {n-1 \choose d-1}$. To see this, observe that since $n-d > 2(k-d)$, equality in (\ref{impline1}) implies that, for any $D \in {[n] \choose d}$, we have either that $\bigtriangledown_\FF(D)$ is a maximum-sized $(d+1)$-star or all of $\bigtriangledown_{[n] \choose k}(D)$. In particular, this implies that $|\bigtriangledown_\FF(D)| = {n-d \choose k-d}$ for all $D \in \GG_d$ and $|\bigtriangledown_\FF(D)| = {n-d-1 \choose k-d-1}$ for all $D \in {n \choose d } \setminus \GG_d$. Suppose for the sake of contradiction that $|\GG_d| < {n-1 \choose d-1}$. Then, we get
\begin{align}
|\FF|{ k \choose d} &= |\GG_d|{n-d \choose k-d} + \left({n \choose d} - |\GG_d|\right){n-d-1 \choose k-d-1} \nonumber\\
&< {n-1 \choose d-1}{n-d \choose k-d} + {n-1 \choose d}{n-d-1 \choose k-d-1} \nonumber\\
&={n-1 \choose k-1} {k \choose d}, \nonumber
\end{align}
which is a contradiction, so $|\GG_d| \ge {n-1 \choose d-1}$. Now, we will show that $|\GG_{d+1}| \ge {n-1 \choose d}$ by a double-counting argument. Observe that if $D \in \GG_d$ and $x \in [n] \setminus D$, then $(D\cup \{x\}) \in \GG_{d+1}$. Furthermore, as noted before, for every $D \in {[n] \choose d} \setminus \GG_d$, we have that $\bigtriangledown_\FF(D)$ is a maximum-sized $(d+1)$-star. Thus, in this case there exists exactly one $x \in [n] \setminus D$ such that $(D \cup \{x\}) \in \GG_{d+1}$. Finally, any element of $\GG_{d+1}$ will be counted in this way precisely $d+1$ times, giving us
\begin{align}
|\GG_{d+1}| &\ge \frac{|\GG_d|(n-d) + \left({n \choose d} - |\GG_d|\right)}{d+1} \nonumber\\
&\ge \frac{{n-1 \choose d-1}(n-d) + {n \choose d} - {n-1 \choose d-1}}{d+1} \nonumber\\
&= {n-1 \choose d}. \nonumber
\end{align}
We show next that $\GG_{d+1}$ must contain no $d$-simplex. To see this, suppose for the sake of contradiction that $\GG_{d+1}$ contains a $d$-simplex $\{D_1,\dots,D_{d+1} \}$. We observe first that $\{D_1,\dots,D_{d+1}\}$ must in fact also be a $d$-cluster. To see the union condition, note that there must be $\{x_1,\dots,x_{d+1}\} \subset [n]$ such that $x_j \in D_i$ for all $i \neq j$. By extension we see $|D_i \setminus \{x_1,\dots,x_{d+1}\}| = 1$, and it follows easily that $|\bigcup_i D_i| \le 2(d+1)$. Next, we choose (not necessarily distinct) $(k-d-1)$-sets $E_1,\dots,E_{d+1} \subseteq [n] \setminus (\bigcup_i D_i)$ such that $\bigcap_i E_i = \emptyset$ and $|\bigcup_i E_i| \le 2(k-d-1)$. Then, because $D_i \in \GG_{d+1}$ for all $1 \le i \le d+1$, it follows that $(D_i \cup E_i) \in \FF$. We claim that $(D_1 \cup E_1),\dots,(D_{d+1} \cup E_{d+1})$ is a $d$-simplex-cluster. To verify this, we check first the union condition

$$\left|\bigcup_i (D_i \cup E_i)\right| = \left|\bigcup_i D_i\right| + \left|\bigcup_i E_i\right| \le 2(d+1) + 2(k-d-1) = 2k,$$
and the first intersection condition
$$\bigcap_i(D_i \cup E_i) = \left(\bigcap_i D_i\right) \cup \left(\bigcap_i  E_i\right) = \emptyset,$$
and finally the second intersection condition
$$\bigcap_{j \neq i}(D_i \cup E_i) = \left(\bigcap_{j \neq i} D_i\right) \cup \left(\bigcap_{j \neq i} E_i\right) \supseteq \left(\bigcap_{j \neq i} D_i\right) \neq \emptyset.$$
However, this contradicts our assumption that $\FF$ is $d$-simplex-cluster-free, so $\GG_{d+1}$ must be $d$-simplex-free. However, the $d+1 = k$ case of Conjecture \ref{simpconj} was resolved by Chv\' atal in \cite{chva}. Since $|\GG_{d+1}| \ge {n-1 \choose d}$, this implies that $\GG_{d+1}$ is a star centered at some $x \in [n]$. We now count the number of elements of $\FF$ that contain $x$ by another double counting argument. For every $D \in \GG_{d+1}$ there will be ${n-d-1 \choose k-d-1}$ elements of $\FF$ that contain it. Furthermore, for any $A \in \FF$ that contains $x$, it will have ${k-1 \choose d}$ subsets of size $d+1$ that contain $x$. From this, we obtain

\begin{align}
|\bigtriangledown_\FF(\{x\})| \ge \frac{|\GG_{d+1}|{n-d-1 \choose k-d-1}}{{k-1 \choose d}} \ge \frac{{n-1 \choose d}{n-d-1 \choose k-d-1}}{{k-1 \choose d}}  = {n-1 \choose k-1}. \nonumber
\end{align}
Thus, $\FF$ is a star centered at $x$. This completes the proof.
\end{proof}

\noindent {\bfseries Acknowledgements:} I would like to thank the Pomona College research circle for introducing me to problems on clusters, and in particular to Shahriar Shahriari and Archer Wheeler for working with me in the early stages. I would also like to thank my advisors J\'{o}zsef Solymosi and Richard Anstee for their support, and in particular to Richard Anstee for many hours spent helping me edit. Finally, I would like to thank P\'{e}ter Frankl for helpful comments and for pointing out additional references, and to the referees for a careful reading and helpful comments.

\bibliographystyle{amsplain}
\bibliography{set_bib}

\providecommand{\bysame}{\leavevmode\hbox to3em{\hrulefill}\thinspace}
\providecommand{\MR}{\relax\ifhmode\unskip\space\fi MR }
\providecommand{\MRhref}[2]{%
  \href{http://www.ams.org/mathscinet-getitem?mr=#1}{#2}
}
\providecommand{\href}[2]{#2}
\begin{thebibliography}{10}

\bibitem{ahlk}
R.~Ahlswede and L.~Khachatrian, \emph{The complete intersection theorem for
  systems of finite sets}, European J. Combin. \textbf{18} (1997), 125--136.

\bibitem{borg1}
P.~Borg, \emph{A short proof of a cross-intersection theorem of {H}ilton},
  Discrete Math. \textbf{309} (2009), 4750--4753.

\bibitem{borg2}
P.~Borg and I.~Leader, \emph{Multiple cross-intersecting families of signed
  sets}, J. Combin. Theory Ser. A. \textbf{117} (2010), 583--588.

\bibitem{chva}
V.~Chv\'{a}tal, \emph{An extremal set-intersection theorem}, J. London Math.
  Soc. \textbf{9} (1974), 355--359.

\bibitem{me}
G.~Currier, \emph{On the d-cluster generalization of {E}rd{\H o}s-{K}o-{R}ado},
  J. Combin. Theory Ser. A. \textbf{182} (2021).

\bibitem{ekr}
P.~Erd{\H o}s, H.~Ko, and R.~Rado, \emph{Intersection theorems for systems of
  finite sets}, Quart. J. Math Oxford Ser. \textbf{12} (1961), 313--320.

\bibitem{spernfam}
P.~Frankl, \emph{On {S}perner families satisfying an additional condition}, J.
  Combin. Theory Ser. A. \textbf{20} (1976), 1--11.

\bibitem{firstfran}
\bysame, \emph{On a problem of {C}hv\'{a}tal and {E}rd{\H o}s on hypergraphs
  containing no generalized simplex}, J. Combin. Theory Ser. A. \textbf{30}
  (1981), 169--182.

\bibitem{franfur}
P.~Frankl and Z.~F\"{u}redi, \emph{A new generalization of the {E}rd{\H
  o}s-{K}o-{R}ado theorem}, Combinatorica \textbf{3} (1983), 341--349.

\bibitem{franfur2}
\bysame, \emph{Exact solution of some {T}ur\'{a}n-type problems}, J. Combin.
  Theory Ser. A. \textbf{45} (1987), 226--262.

\bibitem{frankup}
P.~Frankl and A.~Kupavskii, \emph{{E}rd{\H o}s-{K}o-{R}ado theorem for $\{0,\pm
  1\}$-vectors}, J. Combin. Theory Ser. A. \textbf{155} (2018), 157--179.

\bibitem{surv}
P.~Frankl and N.~Tokushige, \emph{Invitation to intersection problems for
  finite sets}, J. Combin. Theory Ser. A. \textbf{144} (2016), 157--211.

\bibitem{katcyc}
G.~Katona, \emph{A simple proof of the {E}rd{\H o}s-{C}hao {K}o-{R}ado
  theorem}, J. Combin. Theory Ser. A. \textbf{13} (1972), 183--184.

\bibitem{kevmub}
P.~Keevash and D.~Mubayi, \emph{Set systems without a simplex or cluster},
  Combinatorica \textbf{30} (2010), 175--200.

\bibitem{kellif}
N.~Keller and N.~Lifshitz, \emph{The junta method for hypergraphs and the
  {E}rd{\H o}s-{C}hv\'{a}tal simplex conjecture}, arXiv preprint
  arXiv:1707.02643 (2017).

\bibitem{sololif}
N.~Lifshitz, \emph{On set systems without a simplex-cluster and the junta
  method}, J. Combin. Theory Ser. A. \textbf{170} (2020).

\bibitem{3mub}
D.~Mubayi, \emph{Erd{\H o}s-{K}o-{R}ado for three sets}, J. Combin. Theory Ser.
  A. \textbf{113} (2006), 547--550.

\bibitem{4mub}
\bysame, \emph{An intersection theorem for four sets}, Adv. in Mathematics
  \textbf{215} (2007), 601--615.

\bibitem{rammub}
D.~Mubayi and R.~Ramadurai, \emph{Set systems with union and intersection
  constraints}, J. Combin. Theory Ser. B. \textbf{99} (2009), 639--642.

\bibitem{vermub}
D.~Mubayi and J.~Verstra\"{e}te, \emph{Proof of a conjecture of {E}rd{\H o}s on
  triangles in set systems}, Combinatorica \textbf{25} (2005), 599--614.

\end{thebibliography}
\end{document}